\documentclass[12 pt]{article}%
\usepackage{amsmath, amsfonts, amsthm, color,latexsym}
\usepackage{amsmath, ulem}
\usepackage{amsfonts}
\usepackage{amssymb}
\usepackage{color, soul}
\usepackage{xcolor}
\usepackage{tikz-cd}
\usepackage[all]{xy}
\usepackage{graphicx}%
\usepackage{float}
\usepackage{indentfirst}
\providecommand{\U}[1]{\protect\rule{.1in}{.1in}}
\allowdisplaybreaks[4]
\newtheorem{theorem}{Theorem}[section]
\newtheorem{proposition}[theorem]{Proposition}
\newtheorem{corollary}[theorem]{Corollary}
\newtheorem{example}[theorem]{Example}

\newtheorem{remark}[theorem]{Remark}

\newtheorem{lemma}[theorem]{Lemma}
\newtheorem{final remark}[theorem]{Final Remark}

\allowdisplaybreaks[4]

\newcommand{\sol}[1]{\operatorname{sol}(#1)}

\newcommand{\cvf}{\overset{\omega}{\rightarrow}}

\newcommand {\R}{\mathbb{R}}

\newcommand {\N} {\mathbb{N}}
\newcommand{\norma}[1]{\| #1 \|}
\newcommand{\conj}[2]{\left \{ {#1} \, : \, {#2} \right \}}

\begin{document}

\title{The weak maximizing property is not inherited by subspaces of the domain
 }
\author{Alessandro Costa\,, 
Vinícius  Miranda and Geivison Ribeiro}
\date{}
\maketitle

\begin{abstract}
We construct a separable reflexive Banach lattice $X$ and a closed sublattice  $X_1\subseteq X$ such that every bounded linear operator from $X$ into $\ell_2$ is compact, whereas
there exists a positive operator from $X_1$ into $\ell_2$ that does not attain its norm and admits a positive maximizing sequence converging weakly to a nonzero vector.
This answers negatively the question posed by Dantas, Jung and Mart\'inez-Cervantes in \cite{dantasjung} concerning
the inheritance of the weak maximizing property by closed subspaces of the domain. It also answers negatively the corresponding question for the positive weak maximizing property and closed sublattices of reflexive Banach lattices.
Canonical lattice complexification yields a counterexample to the original question over the complex field as well.
\end{abstract}


\section{Introduction}

A pair of Banach spaces $(E, F)$ is said to have the {\it weak maximizing property} (WMP, in short) if every bounded linear operator $T: E \to F$ admitting  a non-weakly null maximizing sequence attains its norm. This property was introduced by Aron, García, Pellegrino and Teixeira in \cite{aron}, and then studied by several different authors, see, e.g, \cite{dantasjung, garcia-lirola, han, jung, mirandapams}. Here a maximizing sequence for $T$ is a sequence $(x_n)\subset S_E$
such that $\|Tx_n\|\to\|T\|$.

In \cite{dantasjung}, Dantas, Jung and Martínez-Cervantes proved that the WMP passes to closed subspaces of the range and to quotients of the domain. In the same manuscript, they posed the following question:

\medskip

\noindent{\bf Question 1} (\cite[Question 4.1]{dantasjung}) Let $E, F$ be Banach spaces and $E_1$ be a closed subspace of $E$. If $(E, F)$ has the WMP, does $(E_1, F)$ also have the WMP?

\medskip

The natural lattice analogue of the WMP was  introduced by Luiz and the second named author (see \cite[Definition 3.1]{luizmiranda}): A pair of Banach lattices $(E, F)$ has the {\it positive weak maximizing property} (WMP$^+$, in short) if every positive operator $T: E \to F$ admitting a positive maximizing sequence which is not weakly null attains its norm. The classical WMP implies WMP$^+$. As noticed in \cite[Examples 3.2]{luizmiranda}, the converse is not true and the positive property does not force the domain to be reflexive. Moreover, it was proved in \cite[Proposition 3.3(b)]{luizmiranda} that if the pair of Banach lattices $(E, F)$ has the WMP$^+$, then $(E_1, F)$ has the WMP$^+$ for every sublattice $E_1 \subset E$ that is the range of a positive contractive projection. However, this fact does not pass to a general sublattice, not even for closed ideals, as we shall see through the following example:

\begin{example} \rm 
The pair $(c, \ell_2)$ has the WMP$^+$, but $(c_0, \ell_2)$ does not. Indeed, it follows from \cite[Examples 3.2(3)]{luizmiranda} that $(c, \ell_2)$ has the WMP$^+$ since $c$ is an AM-space with an order unit. On the other hand, 
consider the positive operator $T: c_0 \to \ell_2$ given by 
$$ T((a_j)_{j=1}^{\infty}) = \left(\frac{\sqrt{3}}{2^j}a_j\right)_{j=1}^{\infty}. $$
For every $a=(a_j)_{j=1}^{\infty}\in B_{c_0}$, we have
$\displaystyle  \norma{Ta}_2^2 =
	3\sum_{j=1}^{\infty}4^{-j}|a_j|^2
	\leq
	3\sum_{j=1}^{\infty}4^{-j}
	=1. $
Moreover, setting $u_n:=\sum_{j=1}^{n}e_j$ for every $n\in\N$, we have $u_n\in S_{c_0}^+$ and
$$
	\lim_{n \to \infty} \norma{Tu_n}_2^2
	= \lim_{n \to \infty} 3\sum_{j=1}^{n}4^{-j}
	= \lim_{n \to \infty} 1-4^{-n} = 1,
	$$
    proving that $\norma{T}=1$ and that $(u_n)_n$ is a positive maximizing sequence for $T$. 
    Since $e_1^*(u_n)=1$
	for every $n\in\N$, this sequence is not weakly null. 	Finally, if $T$ attained its norm at some
	$a\in S_{c_0}$, then
	$\displaystyle 3\sum_{j=1}^{\infty}4^{-j}(1-|a_j|^2)=0.$
	Each summand is nonnegative, so $|a_j|=1$ for every
	$j\in\N$, contradicting $a\in c_0$. Therefore, $T$
	does not attain its norm, and $(c_0,\ell_2)$ fails
	the WMP$^+$.
\end{example}

The example above makes a reflexivity assumption a natural next step: 

\medskip

\noindent{\bf Question 2:} Let $E, F$ be Banach lattices, with $E$ reflexive, and let $E_1$ be a closed sublattice of $E$. If $(E, F)$ has the WMP$^+$, does $(E_1, F)$ also have the WMP$^+$?

\medskip

Our main result is the following: 

\medskip

\noindent {\bf Main Theorem} {\it There exist a separable reflexive Banach lattice $X$ and a closed sublattice $X_1 \subset X$ such that
$(X,\ell_2)$
has the {\rm WMP}, but $(X_1,\ell_2)$ fails to have the {\rm WMP}$^+$.}

\medskip

This result solves Question 2 and also the real case of Question 1. The complex solution for Question 1 will then be obtained in the following corollary:

\begin{corollary} \label{casocomplexo}
Let $X_{\mathbb C}$ and $(X_1)_{\mathbb C}$
be the canonical lattice complexifications of the real Banach lattices $X$ and $X_1$, respectively. Then $(X_{\mathbb C}, \ell_2(\mathbb C))$ has the {\rm WMP}, but $((X_1)_{\mathbb C}, \ell_2({\mathbb C}))$ does not have the {\rm WMP}.
\end{corollary}

Throughout this paper, we reserve the symbols $X$ and $X_1$
for the two specific Banach lattices constructed below,
whereas $E$ and $F$ denote arbitrary Banach spaces or Banach
lattices, depending on the context.
We refer the reader to \cite{aliprantis, meyer} for background on Banach lattices and to \cite{fabian} for Banach space theory.

\section{The construction and the proofs}

Throughout this section, the spaces and operators are real, except in the proof of Corollary \ref{casocomplexo}. We write $B_E$ and $S_E$ for the closed unit ball and the unit sphere of a Banach space $E$. For a subset $A$ in a Banach space $E$, $\operatorname{co}(A)$ denotes the convex hull of $A$. If $A$ is a subset of a Banach lattice, then $\sol{A}$ denotes the solid hull  of $A$, and $A^+$ denotes the set of all elements in $A$ that are positive.

Before giving the details, let us explain the role of each step. We first construct an auxiliary space $X_T$, with coordinates indexed by a tree, and a weakly compact set $K\subset X_T^+$. We then apply the construction of Argyros and Raikoftsalis \cite{argyros} to obtain a space $Z$ with additional properties. Starting from $K$, now regarded as a subset of $Z$, we obtain a solid, convex, weakly compact set $W$. The Davis--Figiel--Johnson--Pe{\l}czy\'nski construction \cite{dfjp}, applied to $Z$ and $W$, produces a reflexive Banach lattice $Y$ such that every bounded linear operator from $\ell_2$ into $Y$ is compact.

Finally, we pass to the dual $Y^*$. An equivalent lattice renorming and an $\ell_\infty$-sum with $\R$ give the spaces $X$ and $X_1$ in the theorem. At several points, we use the Banach lattice structure determined by a $1$-unconditional basis (see, e.g., \cite{lind}).

\subsection{The construction of the auxiliary space $Y$}\label{subsec:F}

Let us begin by describing in detail the construction of a tree 
whose nodes record successive choices from finite subsets of $(0,1]$. For each $n\in\N$, set
$$D_n=\conj{\frac{j}{8^n}}{j=1,\ldots,8^n}.$$
For example,
$D_1=\left\{\frac18,\frac28,\ldots,1\right\}, \,
D_2=\left\{\frac1{64},\frac2{64},\ldots,1\right\},$
and
$D_3=\left\{\frac1{512},\frac2{512},\ldots,1\right\}.$
Each $D_n$ consists of $8^n$ equally spaced points, with distance $8^{-n}$ between consecutive points:
\begin{figure}[H] 
	\centering
	\begin{tikzpicture}[x=1cm,y=1cm,every node/.style={font=\small}]
		\foreach \yy/\dd/\den in {0/1/8,-1.3/2/64,-2.6/3/512}{
			\node[anchor=east] at (-0.5,\yy) {$D_{\dd}:$};
			\draw (0,\yy)--(3,\yy); \draw (4.2,\yy)--(7.2,\yy);
			\foreach \xx in {0,1.5,3,5.7,7.2}{\fill (\xx,\yy) circle (1.5pt);}
			\node[below=5pt] at (0,\yy) {$\frac1{\den}$};
			\node[below=5pt] at (1.5,\yy) {$\frac2{\den}$};
			\node[below=5pt] at (3,\yy) {$\frac3{\den}$};
			\node at (3.6,\yy) {$\cdots$};
			\node[below=5pt] at (5.7,\yy) {$1-\frac1{\den}$};
			\node[below=5pt] at (7.2,\yy) {$1$};}
	\end{tikzpicture}
	\caption{The sets $D_1,D_2,D_3$, displayed schematically.}\label{fig:grids}
\end{figure}

With these sets at hand, define
$T_n=D_1\times\cdots\times D_n.$
The set $T_n$ will be the $n$th level of our tree. A node at this level is a sequence $t=(a_1,\ldots,a_n)$ with $a_k\in D_k$ for $k=1,\ldots,n$. For instance,
$$\displaystyle \left(\frac18\right)\in T_1,\qquad
\left(\frac18,\frac2{64}\right)\in T_2,\qquad
\left(\frac18,\frac2{64},\frac3{512}\right)\in T_3.$$
Each level is finite, with
$\displaystyle |T_n|=\prod_{k=1}^n8^k=8^{n(n+1)/2}.$
Set $T=\bigcup_{n=1}^{\infty}T_n$. To order this set, let $s=(a_1,\ldots,a_m)$ and $t=(b_1,\ldots,b_n)$, and write $s\preceq t$ if $m\leq n$ and $a_j=b_j$ for every $j=1,\ldots,m$. Thus, $s\preceq t$ means that $s$ is an initial segment of $t$. For illustration, we depict below a portion of the first 
three levels of $T$, where each edge joins a node to one 
of its immediate successors.

\begin{figure}[H]
	\centering
	\begin{tikzpicture}[x=1cm,y=1cm,every node/.style={font=\small},edge/.style={thin}]
		\node (a) at (0,0) {$\left(\frac18\right)$};
		\node at (3,0) {$\left(\frac28\right)$};
		\node at (5,0) {$\cdots$}; \node at (7,0) {$(1)$};
		\node (b1) at (-2,-1.7) {$\left(\frac18,\frac1{64}\right)$};
		\node (b2) at (1,-1.7) {$\left(\frac18,\frac2{64}\right)$};
		\node (bd) at (3.5,-1.7) {$\cdots$};
		\node (b3) at (6,-1.7) {$\left(\frac18,1\right)$};
		\draw[edge] (a)--(b1); \draw[edge] (a)--(b2);
		\draw[edge,dotted] (a)--(bd); \draw[edge] (a)--(b3);
		\node (c1) at (-2,-3.4) {$\left(\frac18,\frac2{64},\frac1{512}\right)$};
		\node (c2) at (1,-3.4) {$\left(\frac18,\frac2{64},\frac2{512}\right)$};
		\node (cd) at (3.5,-3.4) {$\cdots$};
		\node (c3) at (6,-3.4) {$\left(\frac18,\frac2{64},1\right)$};
		\draw[edge] (b2)--(c1); \draw[edge] (b2)--(c2);
		\draw[edge,dotted] (b2)--(cd); \draw[edge] (b2)--(c3);
	\end{tikzpicture}
	\label{fig:tree}
\end{figure}

The immediate successors of $(a_1,\ldots,a_n)$ are precisely the nodes
$$\left(a_1,\ldots,a_n,\frac{j}{8^{n+1}}\right),\qquad j=1,\ldots,8^{n+1}.$$
Thus, $T$ is countable, every node has finitely many  immediate successors, and every node belongs to an infinite branch. Every sequence $(b_n)_n$ with $b_n\in D_n$ determines a branch through its initial segments:
$$(b_1)\preceq(b_1,b_2)\preceq(b_1,b_2,b_3)\preceq\cdots.$$

Since $T$ is a countable union of finite sets, we can enumerate its nodes. Fix a bijection $h:T\to\N$ as follows. List all nodes of $T_1$ first, then those of $T_2$, and so on. Within each level, order the nodes by their last coordinate. Resolve ties lexicographically using the preceding coordinates. On the first level, simply use increasing order. This enumeration satisfies \cite[Definition 2.2]{argyros}. 
We should distinguish the tree order from the order of the enumeration. Two nodes may appear consecutively in the enumeration without either being an initial segment of the other. Unless otherwise specified, intervals and block sequences refer to $h$. In particular, a finite interval has the form $h^{-1}(\{p,p+1,\ldots,q\})$, with $p\leq q$.

\par\medskip
Let us now pass to functions on the tree. Let $c_{00}(T)$ be the vector space of all finitely supported functions $x:T\to\R$. For $t\in T$, define $e_t$ by $e_t(t)=1$ and $e_t(s)=0$ for $s\neq t$. Every $x\in c_{00}(T)$ can then be written as $x=\sum_{t\in T}x(t)e_t$.
Equip this space with the norm
$$\norma{x}_T=\left(\sum_{n=1}^{\infty}\left(\sum_{t\in T_n}|x(t)|\right)^2\right)^{1/2},$$
and let $X_T$ be its completion. 

For each $n$, the restriction $x|_{T_n}:T_n\to\R$ is given by $(x|_{T_n})(t)=x(t)$. Since $T_n$ is finite,
$ \displaystyle \sum_{t\in T_n}|(x|_{T_n})(t)|=\sum_{t\in T_n}|x(t)|<\infty.$
Hence $x|_{T_n}\in\ell_1(T_n)$ and
$\displaystyle \norma{x|_{T_n}}_{\ell_1(T_n)}=\sum_{t\in T_n}|x(t)|$, which implies that $x\mapsto(x|_{T_n})_{n=1}^{\infty}$ is a linear isometry from $c_{00}(T)$ into $\ell_2(\oplus_{n=1}^{\infty}\ell_1(T_n))$.
Conversely, given a finitely supported sequence $(v_n)_n$ with $v_n\in\ell_1(T_n)$, define $x(t)=v_n(t)$ for $t\in T_n$. The levels are disjoint and finite, so $x\in c_{00}(T)$. Thus the range is exactly the set of finitely supported sequences in this sum. By density, the map extends to an isometric isomorphism from $X_T$ onto the entire $\ell_2$-sum.
We may therefore identify $X_T$ with
$$\conj{x:T\to\R}{\sum_{n=1}^{\infty}\left(\sum_{t\in T_n}|x(t)|\right)^2<\infty}.$$
Each $\ell_1(T_n)$ is finite-dimensional and hence reflexive. An $\ell_2$-sum of reflexive spaces is reflexive, so $X_T$ is reflexive. Moreover, with respect to the enumeration in $T$, the canonical vectors
$(e_t)_{t\in T}$ form a normalized $1$-unconditional Schauder basis
of $X_T$. Therefore, $X_T$ admits a Banach lattice structure provided by this basis.

For $A\subseteq T$, let $P_A$ denote the coordinate projection onto $A$. For every finitely supported scalar family $(\lambda_t)_{t\in T}$,
\begin{align*}
	\norma{\sum_{t\in A}\lambda_te_t}_T
	&=\left(\sum_n\left(\sum_{t\in T_n\cap A}|\lambda_t|\right)^2\right)^{1/2}\leq\left(\sum_n\left(\sum_{t\in T_n}|\lambda_t|\right)^2\right)^{1/2}
	=\norma{\sum_{t\in T}\lambda_te_t}_T.
\end{align*}
By continuity, this holds for every $x\in X_T$. Recall that $A$ is \textit{segment complete} if $s,t\in A$ and $s\preceq r\preceq t$ imply $r\in A$. The inequality holds, in particular, for all such sets. Thus the norm is SC-unconditional in the sense of \cite[Definition 3.1]{argyros}.

\par\medskip

Fix a node $t=(a_1,\ldots,a_n)\in T_n$. For each $j=1,\ldots,n$, let
$t|_j=(a_1,\ldots,a_j)$. Put $\psi(t)=a_n$. Since $\psi(t|_j)=a_j$, we can consider the vector $\displaystyle k_t=\sum_{s\preceq t}\psi(s)e_s
=\sum_{j=1}^{n}a_j e_{t|_j}.$
For $k_t$, the sum at level $T_j$ is simply
$a_j$ when $1\le j\le n$, and it is zero when $j>n$. Hence
\[
\|k_t\|_T^2
=\sum_{j=1}^{\infty}
\left(\sum_{s\in T_j}|k_t(s)|\right)^2
=\sum_{j=1}^{n}a_j^2.
\]
We now use the enumeration $h$ fixed above. An interval $I$ of this
enumeration consists of nodes occupying consecutive positions in the
list determined by $h$; these nodes need not lie on the same path.
The coordinate projection $P_I$ retains the coordinates indexed by
$I$ and sets all other coordinates equal to zero. Thus,
\[
P_I k_t
=\sum_{\substack{1\le j\le n\\ t|_j\in I}}a_j e_{t|_j}.
\]
Following Definition~3.1 of~[2], with $C=2$, we first select the vectors
$k_t$ with $\|k_t\|_T\le2$ and then take their projections onto intervals
of the enumeration. This gives
\[
K_2=\bigl\{P_I k_t : t\in T,\ \|k_t\|_T\le2,\
I\text{ is an interval of the enumeration}\bigr\}.
\]
Let $K=\overline{K_2}^{\,\|\cdot\|_T}$. Each $k_t$ has nonnegative
coordinates, and deleting coordinates cannot increase the norm. It follows that $K_2\subseteq X_T^+\cap 2B_{X_T}$. Since this
intersection is norm closed, it also contains the norm closure of
$K_2$. Hence $K\subseteq X_T^+\cap 2B_{X_T}$. Notice also that $0\in K_2$: if $t\in T_1$, then $k_t$
is supported on $\{t\}$, so $P_I k_t=0$ for any interval $I$ that does
not contain $t$.
Finally, for each $m\in\mathbb N$, let $P_m$ denote the coordinate
projection onto the first $m$ levels of the tree. Since the enumeration
proceeds level by level, $P_m$ retains precisely the first
$|T_1|+\cdots+|T_m|$ coordinates. For every $x\in X_T$, we have
\[
\|x-P_mx\|_T^2
=\sum_{j=m+1}^{\infty}
\left(\sum_{s\in T_j}|x(s)|\right)^2.
\]
Note that the right-hand side is the tail of the convergent series defining
$\|x\|_T^2$. Therefore $P_mx\to x$ in norm.

\begin{lemma}\label{lemaKiscomp}
	The set $K$ is a weakly compact subset of $X_T^+$.
\end{lemma}

\begin{proof}
    We already know that $K\subseteq X_T^+\cap 2B_{X_T}$. Since $X_T$ is reflexive, it suffices to show that $K$ is weakly closed. Fix $m \in \N$. The set $P_mK_2$ is finite: there are finitely many nodes in the first $m$ levels, and at each node $t$ the coordinate of a vector in $P_mK_2$ is either $0$ or $\psi(t)$. Moreover, $P_mK_2\subseteq K_2$, since applying $P_m$ to $P_Ik_s$ gives $P_{I\cap(\bigcup_{n=1}^mT_n)}k_s.$
The intersection is an interval or is empty; in the latter case we obtain $0\in K_2$. By continuity and finiteness,
$$P_mK\subseteq\overline{P_mK_2}=P_mK_2.$$ The reverse inclusion follows from $K_2\subseteq K$, so $P_mK=P_mK_2\subseteq K$.
	Finally, let $x\in\overline K^{\,w}$. Since $P_m$ is weakly continuous and $P_mK$ is finite, hence weakly closed,
$$P_mx\in\overline{P_mK}^{\,w}=P_mK\subseteq K.$$
	As $P_mx\to x$ in norm and $K$ is norm closed, $x\in K$.
\end{proof}

\par\medskip
We can now apply \cite[Theorem 4.1]{argyros}. The tree satisfies the assumptions of that construction, the norm is SC-unconditional, $X_T$ is reflexive, and $K$ is the weakly compact set associated with $\psi$ and $C=2$ in \cite[Definition 3.1]{argyros}.
We obtain the space $Z=\mathfrak X_\xi$, for a countable ordinal $\xi\geq1$ chosen so that $X_T$ admits no $\ell_1^\xi$-spreading model; see \cite[Lemma 4.13 and Subsection 4.2]{argyros}. The space $Z$ is the completion of $c_{00}(T)$ under a norm satisfying $\norma{x}_T\leq\norma{x}_Z$. The identity extends to a bounded map $I:Z\to X_T$. This map is injective: it preserves all coordinates, and the coordinate vectors form a Schauder basis of $Z$ by \cite[Remark 4.16]{argyros}. \cite[Theorem 4.1]{argyros} also gives that $K\subseteq I(Z)$ and that 
$$B_0=\overline{\operatorname{co}}^{\norma{\cdot}_Z}\bigl(I^{-1}(K)\cup I^{-1}(-K)\bigr)$$
is weakly compact in $Z$. Since $K$ is weakly closed in $X_T$ and $I$ is weak-to-weak continuous, $I^{-1}(K)$ is weakly closed in $Z$. Being contained in $B_0$, it is weakly compact. From now on, we denote this subset of $Z$ simply by $K$.

We next verify that $Z$ has the lattice structure we need. Its canonical basis is bimonotone by \cite[Remark 4.16]{argyros}; in our case, it is also $1$-unconditional. Let $G_1$ and $G_\xi$ be the norming sets in \cite[Subsection 4.2 and Definition 4.15]{argyros}. For a finitely supported functional $f=\sum_tb_te_t^*$ and signs $\varepsilon_t\in\{-1,1\}$, write $f_\varepsilon=\sum_t\varepsilon_tb_te_t^*$. Since sign changes are isometries of $X_T$,
\begin{align*}
	\norma{f_\varepsilon}_{X_T^*}
	&=\sup_{\norma{x}_T\leq1}\left|\sum_t\varepsilon_tb_tx(t)\right|=\sup_{\norma{y}_T\leq1}\left|\sum_tb_ty(t)\right|
	=\norma{f}_{X_T^*}.
\end{align*}
The definition of $G_1$ therefore gives $f_\varepsilon\in G_1$ whenever $f\in G_1$.
The admissibility conditions defining $G_\xi$ depend only on supports, and
$\displaystyle \left(\frac12\sum_{j=1}^df_j\right)_\varepsilon=\frac12\sum_{j=1}^d(f_j)_\varepsilon.$
Induction on the construction of $G_\xi$ gives $f_\varepsilon\in G_\xi$ whenever $f\in G_\xi$. Applying the same signs twice recovers $f$, so this operation is a bijection of $G_\xi$ onto itself. Consequently,
\begin{align*}
	\norma{\sum_t\varepsilon_ta_te_t}_Z
	&=\sup_{f\in G_\xi}f\left(\sum_t\varepsilon_ta_te_t\right)=\sup_{f\in G_\xi}f_\varepsilon\left(\sum_ta_te_t\right)
	=\norma{\sum_ta_te_t}_Z.
\end{align*}
Thus $Z$ is a Banach lattice under the coordinatewise order, and every coordinate projection is contractive.

\begin{lemma}\label{wiswc}
	The set $W=\sol{\overline{\operatorname{co}}^{\norma{\cdot}_Z}(K)}$ is solid, convex and weakly compact in $Z$.
\end{lemma}
\begin{proof}
    It is immediate that $W$ is convex and solid. To prove that $W$ is weakly compact, set $C=\overline{\operatorname{co}}^{\norma{\cdot}_Z}(K)$. Since $K$ is weakly compact in $Z$, $C$ is a weakly compact subset of $Z$ contained in $Z^+$. However, as the order of $Z$ is given by an $1$-unconditional Schauder basis, the norm of $Z$ is order continuous (see, e.g., \cite[Proposition 3.4]{lugarmir}). Thus, by \cite[Theorem 4.39]{aliprantis}, $W$ is relatively weakly compact. 
It remains to show that $W$ is weakly closed. Let  $(x_\alpha)$ be a net in $W$ converging weakly to $x$,
and choose $z_\alpha\in C$ with $|x_\alpha|\leq z_\alpha$. Since $C$ is weakly compact, passing to a subnet, we may assume that $z_\alpha \cvf z \in C$.
Since $z_\alpha\pm x_\alpha\in Z^+$ and $Z^+$ is weakly
closed, we obtain $z\pm x\in Z^+$.
Hence $|x|\leq z$, so $x\in W$.
Therefore $W$ is weakly compact.
\end{proof}

\par\medskip
Fix $1<p<2$. We apply the $\ell_p$ version of the DFJP interpolation construction to $Z$ and $W$; see \cite[Section 5 and Remark 5.8]{argyros} and \cite[p.~281]{maligrandaquevedo}. We use the general interpolation construction here: our solid set $W$ need not coincide with the closed convex hull of $K\cup(-K)$ used in \cite{argyros}. The compactness property required below will be proved separately.
For $n\in\N$, set $U_n=2^nW+2^{-n}B_Z$ and let $\norma{\cdot}_n$ be its Minkowski functional:
$$\norma{z}_n=\inf\conj{\lambda>0}{z\in\lambda U_n}.$$
Each $U_n$ is convex, balanced and solid. It is also weakly closed,
since it is the sum of a weakly compact set and a weakly closed set.
Choosing $R>0$ so that $W\subseteq RB_Z$, we have
$$2^{-n}B_Z\subseteq U_n\subseteq (2^nR+2^{-n})B_Z.$$
Thus $\norma{\cdot}_n$ is an equivalent lattice norm on $Z$, with
$\displaystyle \frac{\norma{z}_Z}{2^nR+2^{-n}}
\leq \norma{z}_n \leq 2^n\norma{z}_Z
$
for every $z \in Z$.
We now set
$$Y=\conj{z\in Z}{
\norma{z}_Y:=\left(\sum_{n=1}^{\infty}\norma{z}_n^p\right)^{1/p}<\infty}.$$
By the DFJP construction, $Y$ is a Banach space, and it is reflexive
because $W$ is weakly compact. To see that $Y$ is an order ideal of $Z$,
take $y\in Y$ and $z\in Z$ with $|z|\leq |y|$. Then
$\norma{z}_n\leq \norma{y}_n$ for every $n$, so $z\in Y$ and
$\norma{z}_Y\leq \norma{y}_Y$. Thus $Y$, with the order inherited
from $Z$, is a Banach lattice. Finally, taking $n=1$ in the estimate
above gives
$$\norma{z}_Z\leq (2R+1/2)\norma{z}_1
\leq (2R+1/2)\norma{z}_Y $$
for every $z \in Y$,
so the inclusion $J\colon Y\to Z$ is bounded.

\begin{proposition}\label{propZ}
	The Banach lattice $Y$ has the following properties: \\
    (1) $Y$ is separable and reflexive.\\
		(2) There is a positive bounded linear surjection $q:Y\to\ell_2$ such that $q^*:\ell_2\to Y^*$ is a lattice embedding.
\end{proposition}
\begin{proof} 
	(1) As noted above, $Y$ is reflexive since $W$ is weakly compact. To prove separability, note that $Z$ has a Schauder basis indexed by the countable set $T$, so it is separable. Each $\norma{\cdot}_n$ is equivalent to $\norma{\cdot}_Z$. Hence $\ell_p(\oplus_{n=1}^{\infty}(Z,\norma{\cdot}_n))$ is separable. The diagonal map $y\mapsto(y,y,\ldots)$ is a linear isometric embedding of $Y$ into this sum, so $Y$ is separable.

	\par\medskip
	(2) Let $J:Y\to Z$ be the inclusion, and let $I:Z\to X_T$ be the map obtained above. Define $\Phi:X_T\to\ell_2$ by
	$$\Phi(x)=\left(\sum_{t\in T_n}x(t)\right)_{n=1}^{\infty}.$$
	 Since
$$\norma{\Phi(x)}_2^2=\sum_n\left|\sum_{t\in T_n}x(t)\right|^2
	\leq\sum_n\left(\sum_{t\in T_n}|x(t)|\right)^2=\norma{x}_T^2$$
    holds for every $x \in X_T$,
	 $\Phi$ is well defined, linear and bounded. It is also positive. Since $I$ and $J$ preserve coordinates, they too are positive, and so is $q=\Phi\circ I\circ J:Y\to\ell_2$.
	Let us verify that $q$ is surjective. First, $W$ is bounded in $Y$, since for $w\in W$ we have $\norma{w}_n\leq2^{-n}$ and therefore
$\displaystyle \norma{w}_Y\leq\left(\sum_{n=1}^{\infty}2^{-np}\right)^{1/p}.$
	Given $a=(a_n)_n\in B_{\ell_2}$, choose $b_n\in D_n$ such that $|a_n|\leq b_n\leq|a_n|+8^{-n}$. An explicit choice is
	$b_n=8^{-n}\max\{1,\lceil8^n|a_n|\rceil\},$
	where $\lceil r\rceil$ is the smallest integer greater than or equal to $r$. Put $t_n=(b_1,\ldots,b_n)$. Since $\norma{(b_n)_n}_2<2$, for every $m$,
	$\displaystyle w_m :=\sum_{n=1}^mb_ne_{t_n}=k_{t_m}\in K.$
	Now, setting $z_m := \displaystyle \sum_{n=1}^ma_ne_{t_n}$, we have $|z_m|\leq w_m$ for every $m \in \N$. Since $K\subseteq W$ and $W$ is solid, $z_m\in W$ for every $m \in \N$. Thus $(z_m)_m$ is a bounded sequence in the reflexive space $Y$, and hence it has a weakly convergent subsequence.  Let $z_{m_j} \cvf z \in Y$. As $q$ is continuous, $q(z_{m_j}) \cvf q(z)$ in $\ell_2$. However, since
$ \displaystyle q(z_{m_j}) = \sum_{n=1}^{m_j} a_n e_n \longrightarrow a $ in $\ell_2$, we conclude that $q(z) = a$, proving that $q$ is surjective. 
It follows that $q^*$ is an isomorphic embedding. To see that it is also a lattice homomorphism, by \cite[Theorem 2.19]{aliprantis} it suffices to verify
	$q([0,y])=[0,q(y)]$ for every $y \in Y^+$.
	One inclusion follows from positivity. 
For the reverse inclusion, take $c=(c_n)_n$ with $0\leq c\leq q(y)$ and set
	$$\alpha_n=\begin{cases}
		\dfrac{c_n}{e_n^*(q(y))},&e_n^*(q(y))>0,\\[5pt]
		0,&e_n^*(q(y))=0.
	\end{cases}$$    
Define $v(t) = \alpha_n y(t)$ for every $t \in T_n$. Since $Z$ is a Banach lattice under the coordinatewise order
and $0\leq\alpha_n\leq1$, the series
$v=\sum_{n=1}^{\infty}\sum_{t\in T_n}\alpha_n y(t)e_t$
converges in $Z$ by comparison with the expansion of $y$,
and $0\leq v\leq y$. As $Y$ is an ideal of $Z$, we have $v \in Y$. Moreover, $q(v) = b$, proving the claim. Therefore, $q^*$ is a lattice homomorphism, and hence $q^*$ is a lattice embedding in the sense of \cite[p. 222]{aliprantis}. 
\end{proof}

We next prove that every bounded linear operator from $\ell_2$
into $Y$ is compact.

\begin{proposition}\label{prop24}
	Every bounded linear operator from $\ell_2$ into $Y$ is compact.
\end{proposition}
\begin{proof}
	Suppose, towards a contradiction, that $T:\ell_2\to Y$ is bounded and noncompact.
    By \cite[Theorem 2.3]{gorokhova}, there are an orthonormal sequence $(u_n)_n$ in $\ell_2$ and $\delta>0$ such that  $\norma{Tu_n}_Y\geq\delta$ for every $n \in \N$. Let $J:Y\to Z$ be the inclusion map.

	We first claim that $\norma{JTu_n}_Z\to0$. Indeed, if $(JTu_n)_n$ does not converge to $0$, then, by passing to a subsequence if necessary, we can assume that $\norma{JTu_n}_Z \geq a$ for every $n \in \N$ and some $a > 0$. Set $\varepsilon=a/64$ and choose $M>\norma{T}$. For all $n,r\in\N$,
$$\norma{JTu_n}_r\leq\norma{Tu_n}_Y\leq\norma{T}<M,$$
so $JTu_n\in M2^rW+M2^{-r}B_Z$. Choose $r$ with $M2^{-r}<\varepsilon$ and put $\lambda=M2^r$. 
Thus, for each $n$, we can choose $\widetilde w_n\in W$ satisfying
	$$\norma{JTu_n-\lambda\widetilde w_n}_Z<\varepsilon.$$
Since $(JTu_n)_n$ is weakly null and each level of the tree is finite, a sliding hump argument gives, after passing to a subsequence, integers $0=m_0<m_1<m_2<\cdots$ such that,
setting $B_n := \displaystyle \bigcup_{k=m_{n-1} + 1}^{m_n} T_k$, we have 
$ \norma{JTu_n - P_{B_n}(JTu_n)}_Z < \varepsilon, $  where $P_{B_n}$ denotes the coordinate projection onto $B_n$. Setting $w_n := P_{B_n} \widetilde{w}_n$ for every $n \in \N$, we have $w_n \in W$ since $W$ is solid. Moreover, 
\begin{align*}
     \norma{JTu_n - \lambda w_n}_Z & \leq \norma{JTu_n - P_{B_n}(JTu_n)}_Z + \norma{P_{B_n}(JTu_n) - P_{B_n}(\lambda \widetilde{w}_n)}_Z < 2 \varepsilon.
\end{align*}
Now, by the definition of $W$, we can choose $(h_n)_n \subset \overline{\operatorname{co}}(K)$ such that $|w_n| \leq h_n$ for every $n \in \N$. Since 
$ P_{B_n}\bigl(\overline{\operatorname{co}}(K)\bigr)
 \subseteq \overline{\operatorname{co}}(K)$ and $|w_n| = P_{B_n}|w_n| \leq P_{B_n} h_n$, replacing $h_n$ by $P_{B_n} h_n$, we may assume that $h_n$ is supported on $B_n$, that is $h_n(t) = 0$ for every $t \in T\setminus B_n$. Approximating $h_n$ by elements of
$\operatorname{co}(K)$ and applying the contractive projection $P_{B_n}$, we can choose
$h_n'\in\operatorname{co}(K)$, supported on $B_n$, such that $\norma{h_n-h_n'}_Z<\dfrac{\varepsilon}{2\lambda}.$ 	These vectors are nonzero: indeed $\lambda\norma{h_n}_Z\geq\lambda\norma{w_n}_Z>a-2\varepsilon$. Thus $(h_n')_n$ is a level-block sequence. Moreover,
	$$h_n'\in\operatorname{co}(K)\subseteq\operatorname{co}(K\cup(-K))=W_\xi^0,$$
with the notation of \cite[Notation 4.24]{argyros}. Each $B_n$ consists of consecutive complete levels and contains the enumeration range of $h_n'$. We may therefore apply \cite[Proposition 4.30]{argyros} with
$\varepsilon/(2\lambda)$ to obtain an infinite set $L\subset\N$ and pairwise incomparable, segment-complete sets $F_n\subseteq B_n$, $n\in L$, such that
$$\norma{h_n'-P_{F_n}h_n'}_Z
 <\frac{\varepsilon}{2\lambda}
 $$
for every $n \in L$.  Passing to the subsequence indexed by $L$ and relabeling all
the corresponding sequences, we may assume that these properties
hold for every $n\in\N$.
Since $|w_n| \leq h_n$ and $\norma{\cdot}_Z$ is a lattice norm, we have
\begin{align*}
    \norma{w_n - P_{F_n}w_n}_Z \leq \norma{h_n - P_{F_n}h_n}_Z \leq \norma{h_n - h_n'}_Z + \norma{h_n'- P_{F_n} h_n'}_Z < \frac{\varepsilon}{\lambda}.
\end{align*}
Consequently, $z_n := \lambda P_{F_n}w_n$ satisfies $\norma{JTu_n - z_n}_Z < 3 \varepsilon$ and $\norma{z_n}_Z \geq a - 3 \varepsilon > a/2$ for every $n \in \N$. Thus, $(z_n)_n$ is a seminormalized level-block sequence with pairwise incomparable supports. 
Following the proof of \cite[Remark 4.21]{argyros}, we claim that
$$\left\|\sum_{n\in A}z_n\right\|_Z
 \geq \frac{a}{8}|A|
 \qquad
 \text{for every nonempty }A\in\mathcal S_\xi,$$
where $\mathcal S_\xi$ is the Schreier family used to define $G_\xi$.  Indeed, since $\|z_n\|_Z=\sup_{f\in G_\xi} f(z_n)>a/2$, we may choose $f_n\in G_\xi$ such that $f_n(z_n)>a/4$. As $F_n$ is segment complete and $z_n$ is supported on $F_n$, \cite[Lemma 4.17]{argyros}  allows us to restrict $f_n$ to $F_n$ without leaving $G_\xi$ or changing its value on $z_n$. Thus, we may assume that
$\operatorname{supp}f_n\subseteq F_n$. These supports are successive and pairwise incomparable. Moreover, $\min h(\operatorname{supp}f_n)\geq n$, so the
spreading property of $\mathcal S_\xi$ implies that
$(f_n)_{n\in A}$, listed in increasing order, is
$(T,\xi)$-admissible whenever $A\in\mathcal S_\xi$ is nonempty (see \cite[Notation 4.14]{argyros}). Therefore $\frac12\sum_{n\in A}f_n\in G_\xi$ and
	$$\norma{\sum_{n\in A}z_n}_Z
	\geq\left(\frac12\sum_{n\in A}f_n\right)\left(\sum_{n\in A}z_n\right)
	=\frac12\sum_{n\in A}f_n(z_n)>\frac a8|A|.$$
	On the other hand, orthonormality gives $\norma{\sum_{n\in A}u_n}_2=|A|^{1/2}$. Thus
	\begin{align*}
		\norma{JT}|A|^{1/2}
		&\geq\norma{JT\left(\sum_{n\in A}u_n\right)}_Z\\
		&\geq\norma{\sum_{n\in A}z_n}_Z-\sum_{n\in A}\norma{JTu_n-z_n}_Z\\
		&\geq\left(\frac a8-3\varepsilon\right)|A|=\frac{5a}{64}|A|.
	\end{align*}
Thus, for every nonempty $A\in\mathcal S_\xi$, since $(u_n)_n$ is an orthonormal sequence,
\begin{align*}
    \norma{JT} |A|^{1/2} &  \geq \norma{JT (\sum_{n \in A} u_n)}_Z \geq \left\|\sum_{n\in A}z_n\right\|_Z
       -\sum_{n\in A}\norma{JTu_n-z_n}_Z\\ 
       & \geq \left(\frac a8-3\varepsilon\right)|A|
 =\frac{5a}{64}|A|.
\end{align*}
Since $\mathcal S_\xi$ contains sets of arbitrarily large
cardinality, this is a contradiction. Therefore, $\norma{JTu_n}_Z\to0$.

Thus, for each fixed $k \in \N$, we have 
$\norma{JTu_n}_k\leq2^k\norma{JTu_n}_Z\to0$
as $n\to\infty$. Consider the isometric embedding
$$\Delta:Y\longrightarrow \ell_p(\oplus_{k=1}^{\infty}(Z,\norma{\cdot}_k)), \qquad \Delta y=(Jy,Jy,\ldots),$$
and let $Q_N$ denote the projection onto the first $N$ coordinates
of the $\ell_p$-sum, with $Q_0=0$. The preceding convergence implies $\norma{Q_N\Delta Tu_n}\to0$ for each fixed $N$. Choose $\eta_j>0$ with $\sum_j\eta_j<\delta/2$. Recursively choose $n_1<n_2<\cdots$ and $0=N_0<N_1<N_2<\cdots$ so that
	$$\norma{Q_{N_{j-1}}\Delta Tu_{n_j}}<\frac{\eta_j}{2},\qquad
	\norma{(\operatorname{Id}-Q_{N_j})\Delta Tu_{n_j}}<\frac{\eta_j}{2}.$$
	We first choose $n_j$ large enough for the first inequality, and then $N_j$ large enough for the second. Setting
$v_j:=(Q_{N_j}-Q_{N_{j-1}})\Delta Tu_{n_j},$
we obtain vectors supported on successive finite intervals
of the $\ell_p$-sum coordinates satisfying
$\displaystyle \sum_{j=1}^{\infty}\norma{\Delta Tu_{n_j}-v_j}
<\sum_{j=1}^{\infty}\eta_j<\frac{\delta}{2}.$
Passing to this subsequence and relabeling, we assume that $n_j=j$.
Since
$\norma{\Delta Tu_n}=\norma{Tu_n}_Y\geq\delta$,
we have $\norma{v_n}\geq\delta/2$, and for every $m\in\N$,
\begin{align*}
 \norma{T}\sqrt m
&\geq\left\|\sum_{n=1}^{m}Tu_n\right\|_{Y}
 =\left\|\sum_{n=1}^{m}\Delta Tu_n\right\|\\
 &\geq
 \left\|\sum_{n=1}^{m}v_n\right\|
 -\sum_{n=1}^{m}\norma{\Delta Tu_n-v_n}\\
 &\geq
\left(\sum_{n=1}^{m}\norma{v_n}^p\right)^{1/p}
 -\frac{\delta}{2}
 \geq\frac{\delta}{2}m^{1/p}-\frac{\delta}{2},
\end{align*}
which is a contradiction since $p<2$. Therefore, $T$ is compact. 
\end{proof}

\subsection{The proof of the Main Theorem}
We need to modify the norm on a sublattice while preserving a lattice norm on the ambient space. The following lemma provides the required extension.

\begin{lemma}\label{lemanorma}
 Let $F$ be a closed sublattice of a Banach lattice $(E, \norma{\cdot}_{E})$. Suppose that $\rho$ is a lattice norm on $F$ such that there are constants $0 < c \leq C$    
    satisfying 
    $$ c \norma{y}_{E} \leq \rho(y) \leq C \norma{y}_{E} $$
    for every $y \in F$. Then 
    $$|\!|\!|x|\!|\!| := \inf \conj{\rho(y) + C \norma{(|x| - y)^+}_E}{y \in F^+} $$
    defines an equivalent lattice norm on $E$ with $|\!|\!|y|\!|\!| = \rho(y)$ for every $y \in F$ and $$c \norma{x}_{E} \leq |\!|\!|x|\!|\!| \leq C \norma{x}_E$$ for every $x \in E$.
\end{lemma}

\begin{proof}
We omit the straightforward proof that $|\!|\!|\cdot|\!|\!| $ defines a lattice norm on $E$ and that $|\!|\!|y|\!|\!| = \rho(y)$ for every $y \in F$. We check that $|\!|\!|\cdot|\!|\!| $ is equivalent to $\norma{\cdot}_E$. Indeed, for every $x \in E$ and $y \in F^+$, we have $|x| \leq y + (|x| - y)^+$, and so
    $$ c \norma{x}_E \leq c (\norma{y}_E + \norma{(|x| - y)^+}_E) \leq \rho(y) + C \norma{(|x| - y)^+}_E. $$
    Taking the infimum yields $c \norma{x}_E \leq |\!|\!|x|\!|\!|$. Finally, taking $y = 0$ in the definition of $|\!|\!|\cdot|\!|\!|$, we get $|\!|\!|x|\!|\!| \leq C \norma{x}_E$, completing the proof.
\end{proof}

\begin{proof}[Proof of the Main Theorem] 
Let $Y$ be the separable reflexive Banach lattice constructed in Subsection 2.1 and let $q: Y \to \ell_2$ be the mapping given by Proposition \ref{propZ}. Let $X_0 = Y^*$ and $F = q^*(\ell_2)$. Since $q^*$ is a lattice embedding, $F$ is a closed sublattice of $X_0$. We shall construct a lattice norm $\rho$ on $F$ satisfying the conditions in Lemma \ref{lemanorma}. To do this, define $D: \ell_2 \to \ell_2$ by $De_n = n^{-1}e_n$ for every $n \in \N$, and put
$$ \rho(y) := \norma{(q^*)^{-1}(y)}_2 + \norma{D(q^*)^{-1}(y)}_2, \qquad y \in F. $$
Since $(q^*)^{-1}: F \to \ell_2$ and $D$ are lattice homomorphisms, $\rho$ is a lattice norm on $F$. Moreover, for every $y \in F$, 
\begin{align*}
    \norma{y}_{X_0} & = \norma{q^*(q^*)^{-1}y}_{X_0} \leq \norma{q^*} \norma{(q^*)^{-1}y}_2 \leq \norma{q^*} \rho(y) \\
    & = \norma{q^*} \left ( \norma{(q^*)^{-1}(y)}_2 + \norma{D(q^*)^{-1}(y)}_2 \right ) \\
    & \overset{\norma{D} \leq 1}{\leq} \norma{q^*} \left ( 2 \norma{(q^*)^{-1}y}_2 \right ) \leq \norma{q^*} 2 \norma{(q^*)^{-1}} \norma{y}_{X_0}.
\end{align*}
Hence
$$ \frac{1}{\norma{q^*}} \norma{y}_{X_0} \leq \rho(y) \leq 2 \norma{(q^*)^{-1}} \norma{y}_{X_0} $$
for every $y \in F$. Since $\dfrac{1}{\norma{q^*}} \leq 2 \norma{(q^*)^{-1}}$, we can apply Lemma \ref{lemanorma} to obtain an equivalent lattice norm $|\!|\!|\cdot|\!|\!|$ on $X_0$ whose restriction to $F$
coincides with $\rho$. Now, we set
$$X := (X_0, |\!|\!|\cdot|\!|\!|) \oplus_\infty \R \quad \text{and} \quad X_1 = (F, \rho) \oplus_\infty \R. $$
From the above construction, $X$ is a Banach lattice and $X_1$ is a closed sublattice of $X$. 
Since $Y$ is separable and reflexive, so its $X_0 = Y^*$. Equivalent renormings and finite direct sums preserve these properties, and therefore $X=(X_0,|\!|\!|\cdot|\!|\!|)\oplus_\infty\R$ is separable and reflexive. We next show that every bounded linear operator $S: X \to \ell_2$ is compact. Indeed, given $S \in \mathcal{L}(X; \ell_2)$, define $S_0 : X_0 \to \ell_2$ by $S_0(x) = S(x,0)$. By equivalence of norms, $S_0$ is bounded with respect to the original dual norm on $X_0$. Identifying $Y^{**}$ with $Y$, Proposition~\ref{prop24}
implies that $S_0^*:\ell_2\to Y$ is compact.
Hence, by Schauder's theorem, $S_0$ is compact,
also when $X_0$ is equipped with $|\!|\!|\cdot|\!|\!|$. Since
$S(x,t) = S_0(x) + t S(0,1)$
holds for every $(x,t) \in X$, $S$ is the sum of a compact operator and a finite-rank operator, and consequently it is compact. 
Therefore, as $X$ is reflexive, $S$ attains its norm. This proves that every bounded linear operator from $X$ into $\ell_2$
attains its norm, and so $(X,\ell_2)$ has the WMP.

To prove that $(X_1, \ell_2)$ fails the WMP$^+$, consider the positive bounded linear operator $T:X_1\to\ell_2$ given by
$$T(y,t)=(q^*)^{-1}(y), \qquad (y,t) \in X_1.$$
For every $(y,t)\in X_1$, we have
$ \norma{T(y,t)}_2 = \norma{(q^*)^{-1}(y)}_2 \leq  \rho(y) \leq \norma{(y,t)}_{X_1}$,
so $\norma{T}\leq1$.

On the other hand, setting $x_n := \displaystyle \left (\frac{q^*(e_n)}{1+n^{-1}}, 1\right )$ for every $n \in \N$, we get that $(x_n)_n$ is a normalized sequence contained in $S_{X_1}^+$ such that
$ \norma{Tx_n}_2 = \frac{1}{1+n^{-1}} \to 1$, proving that $\norma{T} = 1$. Moreover, since $x_n \cvf (0,1)$ in $X_1$, $(x_n)_n$ is a positive non-weakly null maximizing sequence for $T$. Finally, for the sake of contradiction, we suppose that  $T$ attains its norm at some
$(y,t)\in S_{X_1}$. Then
$$ 1 = \norma{T(y,t)}_2 = \norma{(q^*)^{-1}(y)}_2 \leq \rho(y) \leq 1. $$
By the definition of $\rho$ forces $\norma{D(q^*)^{-1}(y)}_2 = 0$, and since $D$ is injective, we obtain $(q^*)^{-1}(y) = 0$, contradicting $\norma{(q^*)^{-1}(y)}_2=1$.
Therefore, $T$ does not attain its norm, and
$(X_1,\ell_2)$ fails the WMP$^+$.
\end{proof}

\begin{remark} \rm  
The sublattice $X_1$ cannot be an ideal in $X$. Indeed, since $X$ is a reflexive Banach lattice, its norm is order continuous. Thus, every closed ideal in $X$ is a projection band (see \cite[Theorem 2.4.4]{meyer}). Its associated band projection $P$ satisfies $0 \leq P \leq {\rm Id}_X$, and hence $P$ is contractive. Since $(X, \ell_2)$ has the WMP, it also has the WMP$^+$. Therefore, if $X_1$ were a closed ideal in $X$, \cite[Proposition 3.3(b)]{luizmiranda} would imply that $(X_1, \ell_2)$ has the WMP$^+$, a contradiction.
\end{remark}

We finish with the complex case, using the canonical lattice complexification described in \cite[Section 2.2]{meyer}.

\begin{proof}[Proof of Corollary~\ref{casocomplexo}] 
 Let $X_{\mathbb C}$ and $(X_1)_{\mathbb C}$ be the canonical complexifications of the real Banach lattices $X$ and $X_1$. Then $X_{\mathbb C}$ is a reflexive complex Banach space and $(X_1)_{\mathbb C}$ is a closed subspace of $X_{\mathbb C}$. Let $S:X_{\mathbb C}\to\ell_2(\mathbb C)$ be bounded and complex linear. For $x\in X$, write $Sx=Ax+iBx$, where $A,B:X\to\ell_2(\R)$ are bounded real-linear operators. Both are compact by the proof of the Main Theorem. Complex linearity gives
	$S(x+iy)=(Ax-By)+i(Bx+Ay).$
The real and imaginary part projections are bounded, so $S$ is compact. As $X_{\mathbb C}$ is reflexive, $S$ attains its norm. Hence $(X_{\mathbb C},\ell_2(\mathbb C))$ has the WMP.

Let $T_{\mathbb C}:(X_1)_{\mathbb C}\to\ell_2(\mathbb C)$
be the complexification of the positive operator $T: X_1 \to \ell_2$ constructed in the proof of the Main Theorem. Since $T$ is positive, $|T_{\mathbb C} z| \leq T |z|$ for every $z \in (X_1)_{\mathbb C}$ (see \cite[Proposition 2.2.6]{meyer}), 
and consequently 
$$ \norma{T_{\mathbb C} z}_2 \leq \norma{T|z|}_2 \leq \norma{T} \norma{|z|}_{X_1} = \norma{T} \norma{z}_{(X_1)_{\mathbb C}} $$
for every $z \in (X_1)_{\mathbb C}$.  Then 
$\norma{T_{\mathbb C}}\leq\norma{T}$.
The reverse inequality follows by restricting $T_{\mathbb C}$ to $X_1$. Hence $\norma{T_{\mathbb C}}=\norma{T}=1$. Besides, since $T$ does not attain its norm, for every $0 \neq z \in (X_1)_{\mathbb C}$, we have
$$ \norma{T_{\mathbb C} z}_2 \leq \norma{T|z|}_2 < \norma{T} \norma{|z|}_{X_1} = \norma{T_{\mathbb C}} \norma{z}_{(X_1)_{\mathbb C}}, $$
which implies that $T_{\mathbb C}$ does not attain its norm.
Finally, considering the sequence $x_n = \displaystyle \left  (\frac{q^*(e_n)}{1 + n^{-1}}, 1\right )$ constructed in the proof of the Main Theorem, the canonical inclusion of $X_1$ into $(X_1)_{\mathbb C}$ yields that $x_n\in S_{(X_1)_{\mathbb C}}$ for every $n \in \N$, and
$$\norma{T_{\mathbb C}x_n}_2
 =\norma{Tx_n}_2
 =\frac{1}{1+n^{-1}}
 \longrightarrow1=\norma{T_{\mathbb C}}.$$
Moreover, once again $(x_n)_n$ is not weakly null in $(X_1)_{\mathbb C}$. In fact, $x_n\cvf(0,1)$ also in $(X_1)_{\mathbb C}$: the real and imaginary parts of the restriction of every complex continuous linear functional to $X_1$ belong to $(X_1)^*$. Therefore, $((X_1)_{\mathbb C},\ell_2(\mathbb C))$ fails the WMP.
\end{proof}

\par\bigskip
\noindent\textbf{Declaration on the use of generative AI.}
The research questions and overall direction of this work were determined by the authors. OpenAI's ChatGPT assisted in developing an initial counterexample. The authors subsequently extended the construction to the setting of Banach lattices, thereby resolving Question~2 and the complex case of Question~1, which were not addressed by the initial example. ChatGPT was also used to assist with mathematical discussions, reference searches, and revisions to the exposition. All arguments, references, and text were carefully checked, corrected, and revised by the authors, who take full responsibility for the content of the manuscript.
\par\medskip
\noindent\textbf{Conflict of interest.} The authors declare that they have no conflict of interest.
\par\medskip
\noindent\textbf{Funding.} The research of Vin\'icius Miranda is supported by FAPESP (Grants 2025/08630-0 and 2023/12916-1) and FAPEMIG (Grant APQ-02622-26).

\par\bigskip
\noindent\begin{minipage}{\linewidth}
	\textsc{Alessandro Costa}\\
	Departamento de Matemática, Universidade Federal do Maranhão\\
	65085-580 -- São Luís -- MA -- Brazil\\
	\textit{Email:} alessandro.mendes@discente.ufma.br
\end{minipage}

\medskip
\noindent\begin{minipage}{\linewidth}
	\textsc{Vinícius Miranda}\\
	Departamento de Matemática, Instituto de Ciências Matemáticas e de Computação\\
	Universidade de São Paulo, 13566-590 -- São Carlos -- SP -- Brazil\\
	\textit{Email:} viniciusmiranda@icmc.usp.br
\end{minipage}

\medskip
\noindent\begin{minipage}{\linewidth}
	\textsc{Geivison Ribeiro}\\
	Departamento de Matemática, Universidade Federal de Sergipe, Itabaiana -- SE -- Brazil\\
	\textit{Email:} geivison@mat.ufs.br
\end{minipage}
\end{document}